\documentclass[12pt]{amsart}

\usepackage{amssymb}
\usepackage{amscd}
\usepackage[all,tips]{xy}
\usepackage[bookmarks=true]{hyperref}
\hypersetup{colorlinks,citecolor=blue,linkcolor=blue}

\newtheorem{thm1}{Theorem}
\newtheorem{theorem}{Theorem}[section]
\newtheorem{lemma}[theorem]{Lemma}

\newtheorem*{claim*}{Claim}
\theoremstyle{definition}

\newcommand\Hy{\mathbb{H}}

\newcommand\isom{\operatorname{Isom}}

\newcommand\Z{\mathbb{Z}}
\newcommand\Q{\mathbb{Q}}
\newcommand\R{\mathbb{R}}
\newcommand\C{\mathbb{C}}

\newcommand\Orb{\mathcal{O}}
\newcommand\Ball{\mathrm{B}(r,x)}

\newcommand\diam{\mathrm{diam}}

\newcommand\vol{\mathrm{vol}}

\DeclareMathOperator{\GL}{GL}

\newcommand{\co}{\colon\thinspace}

\begin{document}
\title{A bound for diameter of arithmetic hyperbolic orbifolds}


\author{Mikhail Belolipetsky}
\address{
IMPA\\  
Estrada Dona Castorina, 110\\
22460-320 Rio de Janeiro, Brazil}
\email{mbel@impa.br}

\begin{abstract}
Let $\Orb$ be a closed $n$-dimensional arithmetic (real or complex) hyperbolic orbifold. We show that the diameter of $\Orb$ is bounded above by
$$\frac{c_1\log\vol(\Orb) + c_2}{h(\Orb)},$$ where $h(\Orb)$ is the Cheeger constant of $\Orb$, $\vol(\Orb)$ is its volume, and constants $c_1$, $c_2$ depend only on $n$. 

\end{abstract}

\maketitle

\section{Introduction}
Let $V$ be a closed real or complex hyperbolic manifold of dimension $n \ge 4$. In \cite{BS87}, Burger and Schroeder proved that $\diam(V) \le \frac{1}{\lambda_1(V)}(\beta_n\log\vol(V)+\alpha_n)$, where $\lambda_1(V)$ is the first non-zero eigenvalue of the Laplacian on $V$ and the constants $\alpha_n$, $\beta_n$ depend only on $n$. It is well known that this inequality does not hold for hyperbolic $2$ and $3$--manifolds (see Remark~(v) in \cite{BS87}). In this article we prove that an inequality of this type still holds if we restrict to \emph{arithmetic} $2$ or $3$--manifolds and extend it to arithmetic orbifolds. Somewhat similar phenomenon is observed for volumes of orbifolds --- compare the discreteness theorems of Wang \cite{Wa72} and Borel \cite{Bor81}. Indeed, the volume formulas of Borel and Prasad play a role in the present argument as well. They enter in a key step of the proof, when we relate the orbifold injectivity radius and the volume. As a result, we show that an arithmetic orbifold $\Orb$ can have a small injectivity radius only if its volume is very large. After this we use an argument due to Brooks \cite{Br92}, which we extend here to orbifolds. The latter extension again requires arithmeticity assumption, which allows us to bound the order of singularities in terms of $\vol(\Orb)$. Dealing with volumes and areas in the presence of singularities requires some results from geometric measure theory. The main theorem of the paper is stated as follows: 

\begin{thm1}\label{thm1}
Let $\Orb$ be a closed arithmetic (real or complex) hyperbolic orbifold of dimension $n\ge 2$. Then the diameter $\diam(\Orb)$, the Cheeger constant $h(\Orb)$ and the volume $\vol(\Orb)$ satisfy the inequality
$$ \diam(\Orb) \le \frac{c_1\log\vol(\Orb)+c_2}{h(\Orb)} $$
with constants $c_1$, $c_2$ depending only on $n$.
\end{thm1}

In dimensions $2$ and $3$ this result is new for orbifolds and for manifolds. In dimensions $n \ge 4$ it is new for orbifolds while for manifolds a similar bound with $h(\Orb)^2$ in the denominator follows from \cite[Theorem~2]{BS87} and Cheeger's inequality $\lambda_1(\Orb) \ge \frac14 h(\Orb)^2$. As an example of a real hyperbolic $n$--orbifold one can consider a sphere $S^n$ with a sufficiently complicated singular set. For example, a figure--$8$ knot in $S^3$ with singularity of order $2$ gives an arithmetic $3$--orbifold. It is not known if arithmetic hyperbolic orbifolds with underlying space $S^n$ exist for large $n$. 

An interesting class of spaces to which Theorem~\ref{thm1} applies are the \emph{congruence arithmetic hyperbolic orbifolds}. These, in particular, include arithmetic orbifolds whose groups are maximal discrete subgroups. The first non-zero eigenvalue of the Laplacian for these orbifolds is uniformly bounded below by a constant which depends only on the dimension (see \cite{BS91} for the real hyperbolic case and \cite{Cl03} in general). By Buser's inequality (generalized to orbifolds) it implies a uniform lower bound for the Cheeger constant \cite{Bus82}. Hence in this case we have an inequality 
$$\diam(\Orb) \le c_1\log\vol(\Orb) +c_2.$$

The methods of this paper can be applied to the other locally symmetric spaces. The higher rank irreducible locally symmetric spaces as well the quaternionic and Cayley hyperbolic orbifolds have Property~T, and thus we can expect here to have a uniform  upper bound for the diameter of the form $c_1\log\vol(\Orb) +c_2$ (cf. \cite[Theorem~1(2)]{BS87} for the quaternionic and Cayley hyperbolic spaces). We leave the details of the general case for a future investigation, in this paper we focus on the real and complex arithmetic hyperbolic orbifolds. 

The paper is organized as follows. In Section 2 we collect various results about metric properties of arithmetic orbifolds. In Section~3 we prove bounds for injectivity radius and volume of small balls in arithmetic hyperbolic orbifolds and finish with the proof of Theorem~\ref{thm1}.

\section{Preliminaries}

\subsection{Hyperbolic orbifolds}
Let $\Gamma$ be a group of isometries of a Riemannian manifold $X$ acting properly discontinuously. If $\Gamma$ is torsion-free, then the quotient space $X/\Gamma$ has a structure of a Riemannian manifold. More generally, the group $\Gamma$ may have finite point-stabilizers, and $X/\Gamma$ is endowed a structure of a \emph{(good) Riemannian orbifold}. An orbifold has an atlas of maps locally identifying it with a quotient of an open set in $X$ by a finite group of isometries. The term orbifold was coined by W.~Thurston in late 1970's. A similar concept was introduced by Satake in \cite{Sat56}, where he used the term \emph{V-manifold}. 

In this paper $X$ will be always a real or complex hyperbolic space: $X = \Hy^n_\R$ or $\Hy^n_\C$.
\subsection{Geodesics and diameter}
Metric properties of Riemannian orbifolds were studied by Borzellino in his thesis \cite{Bor92}. In particular, he showed that there is a natural distance $d$ on $\Orb$ with which $(\Orb,d)$ becomes a length space, and if $(\Orb, d)$ is complete, any two points can be joined by a minimal geodesic realizing the distance between them \cite[Theorem~40, p.~21]{Bor92}. The singular set $\Sigma$ of a good Riemannian orbifold $\Orb$ is locally convex \cite[Proposition~1, p.~31]{Bor92}. Moreover, a geodesic segment cannot
pass through $\Sigma$ unless it starts and/or ends there (see \cite[Theorem~3, p.~32]{Bor92}). A consequence of this fact is that the complement of $\Sigma$ in $\Orb$ is convex as all points in $\Orb\setminus\Sigma$ can be joined by some segment, and $\Sigma$ cannot disconnect $\Orb$.
We define the \emph{diameter} of $\Orb$ as the supremum of the distances between points in $\Orb$. For a closed $\Orb$ the supremum is achieved and there is a geodesic joining $x, y \in \Orb$ whose length is equal to $\diam(\Orb)$.

\subsection{Injectivity radius}
Let $\Orb =  X/\Gamma$ be a closed hyperbolic orbifold and let $\pi \co X \to \Orb$ be the covering map. The elements of the group $\Gamma$ fall into two types: \emph{elliptic} are those which have fixed points in $X$ and \emph{hyperbolic} are those which act freely. For a hyperbolic isometry $\gamma \in \Gamma$ its \emph{displacement at $x \in X$} is defined by $\ell(\gamma, x) = d(x, \gamma x)$ and the \emph{displacement} of $\gamma$ (also called its \emph{translation length}) is 
\[\ell(\gamma) = \inf_{x \in X} \ell(\gamma, x).\]
It is equal to the displacement of $\gamma$ at the points of its \emph{axis}. We will define the orbifold \emph{injectivity radius} by $r_{\mathrm{inj}}(\Orb) = \inf \{\frac12 \ell(\gamma)\}$, where the infimum is taken over all hyperbolic elements $\gamma \in \Gamma$. It is equal to half of the smallest length of a closed geodesic in $\Orb$. When $\Orb$ is a manifold, this definition is equivalent to the usual definition of the injectivity radius as the supremum of $r$ such that any point $p\in\Orb$ admits an embedded ball $\mathrm{B}(p, r) \subset \Orb$. This is not the case in general; the points in the singular set only admit embedded folded balls (cf. \cite{Sam13}).

\subsection{Cheeger's constant} 

We define the \emph{Cheeger constant} $h(\Orb)$ of a closed orbifold $\Orb$ by 
$$ h(\Orb) := \inf \Big( \frac{\mathrm{area}(\partial A)}{\min\{ \vol(A), \vol(\Orb\setminus A)\}}\Big),$$
where $A \subseteq X$ is an open subset with Hausdorff measurable boundary $\partial A$.

By Toponogov's theorem for orbifolds \cite[Theorem~1, p.~28]{Bor92}, a hyperbolic $n$--orbifold is an Alexandrov space with curvature bounded from below. In \cite{OS94}, Otsu and Shioya proved that the singular set $\Sigma$ has Hausdorff dimension $\le n-1$. Moreover, they showed that there exists a $C^0$--Riemannian structure on $\Orb\setminus\Sigma \subset \Orb$ satisfying the following:\
\begin{itemize}
\item[(1)] There exists $\Orb_0 \subset \Orb\setminus\Sigma$ such that $\Orb \setminus \Orb_0$ is of $n$-dimensional Hausdorff measure zero and that the Riemannian structure is $C^{1/2}$--continuous on $\Orb_0$.
\item[(2)] The metric structure on $\Orb\setminus \Sigma$ induced from the Riemannian structure coincides with the original metric of $\Orb$.
\end{itemize}
It follows that we can also compute the volumes in the definition of Cheeger's constant using the  Otsu--Shioya metric. Concerning the area of the boundaries, when $\Orb$ is orientable the singular set has codimension $2$ and we can again use the metric. When $\Orb$ is non-orientable the singular hypersubsets belong to the boundary of $\Orb$ and so do not enter into the formula. We will not use these facts in the proof of the theorem but they are important for understanding the metric structure of hyperbolic orbifolds. 

\subsection{Arithmetic orbifolds} Let $H$ be a linear semisimple Lie group with trivial center, in our case we have $H = \mathrm{Isom}(X)$. Let $\mathrm{G}$ be an algebraic group defined over a number field $k$ such that $\mathrm{G}(k \otimes_\Q \R)$ is isogenous to $H\times K$, where $K$ is a compact Lie group. Consider a natural projection $\phi \co\mathrm{G}(k \otimes_\Q \R) \to H$. The image of the group of $k$-integral points $\mathrm{G}(\mathcal{O}_k)$ and all subgroups $\Gamma < H$ which are commensurable with it are called \emph{arithmetic subgroups of $H$ defined over $k$}. Borel and Harish--Chandra proved that arithmetic subgroups are lattices, i.e., they are discrete and have finite covolume in $H$ \cite{BHC62}. Their quotient spaces $\Orb = X/\Gamma$ are called \emph{arithmetic orbifolds}, they come together with an associated number field $k$, the \emph{field of definition}. Number--theoretical properties of the field $k$ can be used to extract important information about geometry of $\Orb$. This principle underlines the result of the present paper.

We refer to \cite{WM15} for a comprehensive introduction to the theory of arithmetic subgroups and their quotient spaces. 

\section{Proof of Theorem 1}

We first prove two lemmas. Similar results for real hyperbolic orbifolds were previously proved in \cite[Section~3]{AB19} (see also \cite{Bel10}). Here we extend them to the complex hyperbolic case. We repeat some details for the readers convenience. 

\begin{lemma}[A bound for injectivity radius] \label{lem1}
Given a closed $n$-dimensional arithmetic hyperbolic orbifold $\Orb = X/\Gamma$ of sufficiently large volume, we have 
\begin{equation*}
r_{\mathrm{inj}}(\Orb) \ge a_1\left(\frac{\log\log\log \vol(\Orb)^{b_1}}{\log\log \vol(\Orb)^{b_1}}\right)^3
\end{equation*}
with the constants $a_1, b_1 > 0$ depending only on $X$.
\end{lemma}

\begin{proof}
Let $\gamma \in \Gamma$ be a hyperbolic transformation.  The eigenvalues of $\gamma$ considered as an element of $\mathrm{SL}(n+1, \C)$ are $e^{\pm c_1\ell(\gamma)}$ and $n-1$ eigenvalues whose absolute value is $1$ with $c_1 = c_1(X)$ (for the real hyperbolic case it is by \cite[Proposition~1(1,4)]{Gr62}, and for the complex hyperbolic case it follows from \cite[Section~3.3.3]{Gol99}).

We would like to relate $e^{c_1\ell(\gamma)}$ to the Mahler measure of a certain polynomial naturally associated to $\gamma$. To this end we can adapt the argument of \cite[Section~10]{Gel04}. Let $H^\circ$ be the identity component of the group $H = \isom(X)$. It is center-free and connected so we can identify it with its adjoint group $\mathrm{Ad}(H^\circ) \le \mathrm{GL}(\mathfrak{g})$, where $\mathfrak{g}$ denotes the Lie algebra of $H$. We have $\Gamma' = \Gamma\cap H^\circ$, a cocompact arithmetic lattice, and $\gamma^2 \in \Gamma'$. Since $\Gamma'$ is arithmetic, there is a compact extension $H^\circ\times K$ of $H^\circ$ and a $\Q$--rational structure on the Lie algebra $\mathfrak{g}\times\mathfrak{k}$ of $H^\circ\times K$, such that $\Gamma$ is the projection to $H^\circ$ of a lattice $\tilde{\Gamma}$, which is contained in $(H^\circ\times K)_\Q$ and commensurable to the group of integral points $(H^\circ \times K)_\Z$ with respect to some $\Q$--base of $(\mathfrak{g}\times\mathfrak{k})_\Q$. By changing the $\Q$--base we can assume that $\tilde{\Gamma}$ is contained in $(H^\circ \times K)_\Z$. Thus the characteristic polynomial $P_{\tilde{\gamma}}$ of any $\tilde{\gamma} \in \tilde{\Gamma}$ is a monic integral polynomial of degree at most $m\deg(k)$, where $k$ is the field of definition of the arithmetic group and the dimension bound $m$ depends only on the type of $H$. Since $K$ is compact, any eigenvalue of $\tilde{\gamma}$ with absolute value different from $1$ is also an eigenvalue of its projection to $H^\circ$. Therefore we have 
$M(P_{\gamma^2}) = M(P_{\tilde{\gamma}^2})$, and hence
\begin{equation*}
e^{c_1\ell(\gamma^2)} = M(P_{\tilde{\gamma}^2}), 
\end{equation*}
where the \textit{Mahler measure} of an integral monic polynomial $P(x)$ of degree $d$ is defined by 
\[ M(P) = \prod_{i=1}^d \mathrm{max}(1, |\theta_i|),\]
where $\theta_1$,\ldots, $\theta_d$ are the roots of $P(x)$.

Hence we have 
\begin{equation*}\label{sec2:eq1}
\ell(\gamma) \ge \frac{1}{2c_1}\log M(P_{\tilde{\gamma}^2}).
\end{equation*}
This implies that $r_{\mathrm{inj}}(\Orb) \ge \min\{\frac{1}{4c_1} \log M(P_{\tilde{\gamma}}) \}$, where the minimum is taken over all $\tilde{\gamma} \in \tilde{\Gamma}$ which project to hyperbolic elements in $\Gamma'$. Moreover, our argument shows that the degrees of the irreducible integral monic polynomials whose Mahler measures appear in this bound satisfy
\begin{equation*}\label{sec2:eq2}
d \le m\deg(k).
\end{equation*}

By Lehmer's conjecture the Mahler measures of non-cyclotomic polynomials are expected to be uniformly bounded away from~$1$. A special case of this conjecture also known as the Margulis conjecture implies a uniform lower bound for the lengths of closed geodesics of arithmetic locally symmetric $n$--dimensional manifolds (see \cite[Section~10]{Gel04}). These conjectures have attracted a lot of interest but still remain wide open. For our estimate we will take advantage of the known quantitative number-theoretical results towards Lehmer's problem.

In \cite{Dob79}, Dobrowolski proved the following lower bound for the Mahler measure:
\begin{equation}\label{sec2:eq3}
\log M(P) \ge c_2\left(\frac{\log\log d}{\log d}\right)^3,
\end{equation}
where $d$ is the degree of the polynomial $P$ and $c_2>0$ is an explicit constant.

We can relate the degree $d$ to the volume by using an important inequality relating the volume of a closed arithmetic orbifold and the degree of its field of definition:
\begin{equation}\label{sec2:eq4}
\mathrm{deg}(k) < c_3\log\vol(\Orb) + c_4.
\end{equation}
For hyperbolic orbifolds of real dimension $n \ge 4$ this inequality follows from \cite[Section~3.3]{Bel07} and Minkowski's bound for discriminant. In dimensions $2$ and $3$ this inequality is a result of Chinburg and Friedman \cite{CF86}, and in the form stated here it can be found in \cite[Section~3]{BGLS10}. In all the cases the proofs are based on the volume formulae for arithmetic quotient spaces. 

For sufficiently large $x$ the function $\frac{\log x}{x}$ is monotonically decreasing, hence for sufficiently large volume we obtain the inequality of the lemma with $a_1 = \frac{c_2}{4c_1}$ and $b_1 = c_3$.
\end{proof}

The points of orbifold at which the injectivity radius is small form a \emph{thin part}. The Margulis lemma implies that hyperbolic orbifolds also have \emph{thick part} where the injectivity radius is bigger than the Margulis constant. We will use Margulis' lemma in the form given in \cite[Section~2.2]{Sam13}. Let $\mu_n > 0$ and $m_n\in \mathbb{N}$ be the constants defined there which depend only on the space $X$, and let  $\varepsilon = \min \{\frac{\mu_n}{2}, \frac{r_{\mathrm{inj}}(\Orb)}{4m_n}\}$. 

\begin{lemma}[A bound for the volumes of balls] \label{lem2}
A ball of radius $r \le \varepsilon$ in an arithmetic hyperbolic orbifold $\Orb = X/\Gamma$ has volume 
$$\vol(\Ball) \ge \frac{v_r}{a_2 \log\vol(\Orb) + b_2},$$
where $v_r$ denotes the volume of a ball of radius $r$ in $X$ and the constants $a_2, b_2 > 0$ depend only on $X$.
\end{lemma}

\begin{proof}
We first bound the order $q$ of finite subgroups $F < \Gamma$ in terms of volume. We do this by applying Margulis lemma to the discrete subgroups of $\mathrm{O}(n)$ or $\mathrm{U}(n)$ as it was done in \cite{ABSW08} in the case of orthogonal groups. Consider a $k$--embedding of $\Gamma$ into $\mathrm{GL}(m, k)$ with $m = n+1$ if $X= \Hy^n_\C$ or $\Hy^n_\R$ with $n$ even, $m = 2(n+1)$ if $X = \Hy^n_\R$ with $n$ is odd and $\neq 7$, and $m = 24$ if $X = \Hy^7_\R$. The existence of such embedding follows from the classification of $\Q$--forms of the orthogonal and unitary groups (see \cite[Section~18.5]{WM15}).

Let $A\in \GL(m,k)$ be a torsion element of order $t\ge2$. Let
$\lambda_1, \ldots,\lambda_m$ be the eigenvalues of $A$. As $A$ is a matrix
over $k$, its eigenvalues split into groups of conjugates under the action of
$\mathrm{Gal}(\bar{k}/k)$. The equality $A^t= \mathrm{Id}$ implies that the eigenvalues are roots of
unity. Let $t_1, \ldots, t_l$ be their orders, so $t={\rm lcm}(t_1, \ldots,
t_l)$. If $\lambda$ is an eigenvalue then all its $\mathrm{Gal}(\bar{k}/k)$-conjugates
are also eigenvalues of $A$, which implies
$$
\phi_k(t_1)+\cdots+\phi_k(t_l) \le m,
$$
where $\phi_k(t)$ denotes a generalized Euler $\phi$-function defined as the
degree over $k$ of the cyclotomic extension $k(\mu_t)$, $\mu_t$ is a primitive
$t$--root of unity.

It is clear that the following inequalities are satisfied for $\phi_k(t)$ and the Euler
$\phi$-function:
$$
\phi(t)/\deg(k)\le \phi_k(t) \le \phi(t).
$$
We have
$$
\phi(t_i)\le m \deg(k), 
$$
which implies
$$
\phi(t) \le \phi(t_1)\cdots\phi(t_l) \le (m \deg(k))^m.
$$

By using the well known inequality $\phi(t)\ge \sqrt{t}/2$, we obtain
\begin{equation}
t\le 4 (m\deg(k))^{2m} \le c_5\deg(k)^{2m},
\end{equation}
where the constant $c_5$ depends only on $m$. 

This is a bound for the order of finite cyclic subgroups of $\Gamma$, which can be used to give an upper bound for $q$. 
By the Margulis lemma, the constant $m_{n}$ has a property that if $F < \mathrm{O}(n)$ or $\mathrm{U}(n)$ is a finite subgroup, then there is an abelian subgroup $A< F$ such that $[F:A] \leq
m_{n}$ (see \cite[Theorem~2.1]{Sam13}). We may find common
complex eigenspaces of the elements of $A$: $U_{1}, \ldots, U_{k}$, and real
eigenspaces $V_{1},\ldots,V_{l}$, where $k/2+l\leq m$, such that $A$ acts on
$U_{i}$ as a cyclic group $A_{i}$, and $A$ acts on $V_{i}$ as $\pm 1$. We may
embed $A$ in $\prod_{i=1}^{k} A_{i} \times \left(\Z/2\Z\right)^{l}$, acting on
$\prod_{i=1}^{k} U_{1} \times \prod_{j=1}^{l} V_{j}$.  Thus $|A| \leq
2^{l}\prod_{i=1}^{k} |A_{i}|$, and by the previous argument, $|A_{i}| \leq
c_{5}\deg(k)^{2m}$, since a generator of $A_{i}$ is the projection of an
element of $A$. Thus we obtain
\begin{equation}
 q \le m_{n}|A| \le c_6\deg(k)^{c_7}, 
\end{equation}
with the constants $c_6$, $c_7 > 0$ depending only on $m$. 

Together with \eqref{sec2:eq4} it implies 
\begin{equation}\label{eq5}
 q \le c_8 \log\vol(\Orb)^{c_7} + c_9, 
\end{equation}
with the constants depending only on $m$, and hence only on $X$. 

We now use the Margulis lemma once again to show that every $\varepsilon$--ball in $X$ maps to $\Orb$ with multiplicity at most $q$ at each point.  The constants $\mu_n$ and $m_n$ in the Margulis lemma have the following property.  For any $x \in X$ and any $t \in \mathbb{R}$, let $\Gamma_t(x)$ denote the subgroup of $\Gamma$ generated by the elements that move $x$ by distance less than $t$.   Then if $t \leq \mu_n$ and if $\Gamma_t(x)$ is infinite, there is an element in $\Gamma$ of infinite order that moves $x$ by distance less than $2m_nt$ \cite[Lemma~2.3]{Sam13}.  We have chosen $\varepsilon$ such that $2\varepsilon \leq \mu_n$ and $2m_n(2\varepsilon) \leq r_{\mathrm{inj}}(\Orb)$.  By definition, every element in $\Gamma$ of displacement less than $r_{\mathrm{inj}}(\Orb)$ has a fixed point and therefore has finite order.  Thus, $\Gamma_{2\varepsilon}(x)$ must be finite for all $x \in X$.  Let $x_1, \ldots, x_k$ be points in some $\varepsilon$--ball in $X$ that all map to the same point of $\Orb$.  Then they are all in the orbit of $x_1$ under $\Gamma_{2\varepsilon}(x_1)$, and $\Gamma_{2\varepsilon}(x_1)$ has order at most $q$ satisfying \eqref{eq5} as we have shown this for every finite subgroup of $\Gamma$. Hence $k \leq q$ and the multiplicity of the map is at most $q$ on the $\varepsilon$-ball. This finishes the proof of the lemma.
\end{proof}

\begin{proof}[Proof of the theorem] 
By the theorems of Wang and Borel \cite{Wa72, Bor81} there are finitely many arithmetic hyperbolic $n$-orbifolds of bounded volume. Hence we can assume that $\vol(\Orb)$ is sufficiently large and compensate the remaining ones by the additive constant. 

We now adapt the argument of \cite{Br92}. 

Let $r_{\mathrm{min}} =  a_1\left(\frac{\log\log\log \vol(\Orb)^{b_1}}{\log\log \vol(\Orb)^{b_1}}\right)^3$ be the bound for injectivity radius from Lemma~\ref{lem1}, and let $r_0 = {r_{\mathrm{min}}}/{4m_n}$. As $\vol(\Orb)$ is large, we have $r_0 \leq \varepsilon$ from Lemma~\ref{lem2}.
Pick a point $x$ in $\Orb$ and consider a ball $\Ball$ of radius $r$ about $x$. If $r \le r_0$, then the volume of the ball is estimated by Lemma~\ref{lem2}. We now want to estimate the volume $V(r) = \vol(\Ball)$ for $r > r_0$. By the coarea formula (cf. \cite[Section~3.2]{Fed69}) and the definition of the Cheeger constant we have 

$$\frac{V'(r)}{V(r)} \ge h(\Orb),$$
as long as $V(r) < \frac12 \vol(\Orb)$. 

By integrating we obtain
$$ V(r) \ge e^{h(\Orb)(r-r_0)} V(r_0) \text{ until } V(r) = \frac12 \vol(\Orb).$$
This will happen when 
$$ r = r_1 = \frac{1}{h(\Orb)}\Big(\log\vol(\Orb) - \log(2) -\log V(r_0)\Big) + r_0. $$
Therefore, for any $x,y \in \Orb$ we have $\mathrm{B}(r_1,x) \cap \mathrm{B}(r_1,y) \neq \emptyset$, and hence $\diam(\Orb) \le 2r_1$. By Lemma~\ref{lem2} and the value of $r_0$ provided by Lemma~\ref{lem1} we conclude that the terms $-\log V(r_0)$ and $r_0$ are much smaller than $\log \vol(\Orb)$. 
\end{proof}

\medskip

{\noindent\bf Acknowledgments.} I would like to thank the referee for careful proofreading of the manuscript and helpful comments. The author is partially supported by CNPq and FAPERJ research grants.

\end{document}